\documentclass[12pt,a4paper]{article}

\usepackage{graphicx,tikz}
\usepackage{latexsym}
\usepackage{amsmath}
\usepackage{amssymb}
\usepackage{array}
\usepackage{arydshln}

\usepackage{cite}
\usepackage{stmaryrd}
\usepackage{enumerate}

\usepackage{hyperref}

\usepackage{color}
\usepackage{lineno}
\usepackage{graphicx}
\usepackage{ae}
\usepackage{amsmath}
\usepackage{amssymb}
\usepackage{latexsym}
\usepackage{url}
\usepackage{epsfig}
\usepackage{mathrsfs}
\usepackage{amsfonts}
\usepackage{amsthm}
\usepackage{bm}

\newcommand{\Cay}{\mathrm{Cay}}

\newtheorem{theorem}{Theorem}[section]

\newtheorem{lemma}[theorem]{Lemma}

\theoremstyle{definition}

\newtheorem{problem}{Problem}

\numberwithin{equation}{section} 
\allowdisplaybreaks    

\def\qed{\hfill$\Box$\vspace{12pt}}

\long\def\delete#1{}

\usepackage{xcolor}
\usepackage[normalem]{ulem}

\usepackage{stfloats}

\begin{document}
\title{Enumeration of dicirculant digraphs}
\author{Jing Wang$^{a,b,c}$, ~Ligong Wang$^{a,b,c}$$^,$\thanks{Supported by the National Natural Science Foundation of China (No. 12271439).},~Xiaogang Liu$^{a,b,c}$$^,$\thanks{Supported by the National Natural Science Foundation of China (No. 12371358) and the Guangdong Basic and Applied Basic Research Foundation (No. 2023A1515010986).}~$^,$\thanks{Corresponding author. Email addresses: wj66@mail.nwpu.edu.cn, lgwangmath@163.com, xiaogliu@nwpu.edu.cn}
\\[2mm]
{\small $^a$School of Mathematics and Statistics,}\\[-0.8ex]
{\small Northwestern Polytechnical University, Xi'an, Shaanxi 710072, P.R.~China}\\
{\small $^b$Research \& Development Institute of Northwestern Polytechnical University in Shenzhen,}\\[-0.8ex]
{\small Shenzhen, Guangdong 518063, P.R. China}\\
{\small $^c$Xi'an-Budapest Joint Research Center for Combinatorics,}\\[-0.8ex]
{\small Northwestern Polytechnical University, Xi'an, Shaanxi 710129, P.R. China}\\
}
\date{}

\date{}

\openup 0.5\jot
\maketitle
\begin{abstract}
Let $T_{4p}=\langle a,b\mid a^{2p}=1,a^p=b^2, b^{-1}ab=a^{-1}\rangle$ be the dicyclic group of order $4p$. A Cayley digraph over $T_{4p}$ is called a dicirculant digraph. In this paper, we calculate the number of (connected) dicirculant digraphs of order $4p$ ($p$ prime) up to isomorphism by using the P{\'o}lya Enumeration Theorem. Moreover, we get the number of (connected) dicirculant digraphs of order $4p$ ($p$ prime) and out-degree $k$ for every $k$.


\end{abstract}

\section{Introduction}
Let $G$ be a group and $S$ a subset of $G$. The \emph{Cayley digraph} $\Cay(G,S)$ is a digraph whose vertex set is $G$ and arc set is $\{\{g,sg\}\mid g\in G, s\in S\}$. If $S=S^{-1}=\left\{s^{-1}\mid s\in S\right\}$ (inverse-closed), then $\Cay(G,S)$ is an undirected graph, which is also called \emph{Cayley graph}.
A Cayley digraph (respectively, Cayley graph) is connected if and only if $S$ generates $G$.

Enumeration of isomorphic Cayley digraphs is an interesting yet difficult problem. Numerous researchers endeavor to investigate this problem. In 1967,
Turner \cite{Turner1967} found that the P{\'o}lya Enumeration Theorem is a suitable tool for determining the number of Cayley digraphs. By employing the P{\'o}lya Enumeration Theorem, Mishna \cite{Mishna2000} obtained the number of Cayley digraphs (Cayley graphs) over cyclic groups, which was also studied in \cite{AlspachM2002,LiskovetsP2000}. Recently, Huang and Huang \cite{HuangH2019,HuangHL2017} counted Cayley digraphs (Cayley graphs) over  diherdal groups $D_{2p}$~($p$ prime). For enumerating the isomorphism classes of some edge-transitive but not arc-transitive Cayley digraphs, we refer the readers to \cite{LiSim2001,Wang1994,Xu1992}.

Let
$$T_{4p}=\langle a,b\mid a^{2p}=1,a^p=b^2, b^{-1}ab=a^{-1}\rangle$$
be a dicyclic group, which is also called a generalized quaternion group. A Cayley digraph over cyclic group is called a circulant digraph and the one over dicyclic group is called a dicirculant digraph. In 2021, Wang, Liu and Feng \cite{WangLIU2021} studied the number of connected dicirculant graphs of order $4p$~($p$ prime). However, the enumeration of dicirculant digraphs of order $4p$~($p$ prime) remains an unsolved problem.

In this paper, we will solve this problem completely. We first calculate the number of dicirculant digraphs of order $4p$~($p$ prime) by using the P{\'o}lya Enumeration Theorem. And we get the number of connected dicirculant digraphs by deleting the  number of circulants and other disconnected graphs. We also count the (connected) dicirculant digraphs of order $4p$~($p$ prime) and  out-degree $k$ for each $k$.
Finally, we list the number of connected dicirculant digraphs of order $4p$~($2\leq p\leq 11$) and out-degree $k~(0\leq k\leq 4p-1)$.

\section{Preliminaries}
In this section, we introduce the P{\'o}lya Enumeration Theorem and some useful results.

Let $G$ be a group and $X$ be a set. An \emph{action} of $G$ on $X$, denoted by $(G,X)$, is a map $G \times X\rightarrow X$ with $(g,x)\rightarrow gx$ such that
\begin{itemize}
\item[\rm $(i)$] $ex=x$ for all $x \in X$, where $e$ is the identity element of $G$,
\item[\rm $(ii)$] $(g_1g_2)x=g_1(g_2x)$ for all $g_1,g_2\in G$ and all $x\in X$.
\end{itemize}
If there is an action of $G$ on $X$, then we say $G$ \emph{acts} on $X$.
Let $H$ be a permutation group on $X~(|X|=n)$. Then $H$ acts on $X$ naturally by defining $hx=h(x)$. Let $b_k(h)$ be the number of cycles of length $k$ in the standard cycle decomposition of $h\in H$, where $k=1,2, \ldots, n$. The cycle type of $h$ is defined as
$$\mathcal{T}(h)=\left(b_1(h), b_2(h), \ldots, b_n(h)\right).$$
Clearly,
$$b_1(h)+2 b_2(h)+\cdots+n b_n(h)=n.$$
The \emph{cycle index} $\mathcal{I}(H, X)$ of the permutation group $H$ acting on $X$ is defined as the following polynomial
\begin{equation}\label{enequation1}
\mathcal{I}(H, X)=P_H\left(x_1, x_2, \ldots, x_n\right)=\frac{1}{|H|} \sum_{h \in H} x_1^{b_1(h)} x_2^{b_2(h)} \cdots x_n^{b_n(h)},
\end{equation}
where $x_1, x_2, \ldots, x_n$ are indeterminates.
Let $A$ and $C$ be finite sets. Denote by
$$
C^A=\{f \mid f: A \rightarrow C\}
$$
the set of all maps from $A$ to $C$. Let $G$ be a permutation group acting on $A$. Then we obtain a group action $\left(H, C^A\right)$ by:
$$
h f=f \circ h^{-1} \text { for every } h \in H \text { and } f \in C^A,
$$
where $f \circ h^{-1}$ denotes the composite of two maps $f$ and $h^{-1}$. Under the group action $\left(H, C^A\right)$, we say two maps in $C^A$ are \emph{$H$-equivalent} if they belong to the same orbit. The P{\'o}lya Enumeration Theorem gives the number of orbits of the group action $\left(H, C^A\right)$.

\begin{lemma}\emph{(P{\'o}lya Enumeration Theorem, see \cite[Chap. 2]{HararyP1973})} \label{enlemma8}
Let $A$ and $C$ be finite sets with $|A|=n$ and $|C|=m$. Let $H$ be a permutation group acting on $A$. Denote by $\mathcal{F}$ the set of all orbits of the group action $\left(H, C^A\right)$. Then
$$
|\mathcal{F}|=P_H(m, m, \ldots, m),
$$
where $P_H\left(x_1, x_2, \ldots, x_n\right)$ is the cycle index of $(H, A)$ defined in (\ref{enequation1}).
\end{lemma}

Let $H$ be a permutation group acting on $A$. Two $k$-subsets $S$ and $T$ of $A$ are said to be \emph{$H$-equivalent} if there exists some $h \in H$ such that $g(S)=T$. The following result enumerates the $H$-equivalent $k$-subsets of $A$.

\begin{lemma}\emph{(see \cite[Chap. 2]{HararyP1973})} \label{enlemma10}
Let $A$ be a finite set with $|A|=n$, and let $H$ be a permutation group acting on $A$. Then the number of $H$-equivalent classes of $k$-subsets of $A$ is equal to the coefficient of $x^k$ in the polynomial $P_H(1+x, 1+$ $\left.x^2, \ldots, 1+x^n\right)$, where $P_H\left(x_1, x_2, \ldots, x_n\right)$ is the cycle index of $(H, A)$ defined in (\ref{enequation1}).
\end{lemma}



During the process of employing the P{\'o}lya Enumeration Theorem to investigate the number of Cayley digraphs, the graph and the related group must respectively possess the property known as DCI-graph and DCI-group. The Cayley digraph $\Cay(G, S)$ is called a \emph{DCI-graph} of $G$ if, for any Cayley digraph $\Cay(G, T)$, whenever $\Cay(G, S) \cong \Cay(G, T)$ we have $\alpha(S)=T$ for some $\alpha \in \mathrm{Aut}(G)$, where $\mathrm{Aut}(G)$ denotes the automorphism group of $G$. A group $G$ is called a \emph{DCI-group} if all Cayley digraphs on $G$ are DCI-graphs. Many DCI-graphs and DCI-groups have been investigated, and the readers can refer to \cite{Adam1967,ElspasT1970,Babai1977,Dobson1995,Dobson1998,Dobson2014,DobsonMS2015,FengK2018,HuangQX2003, HuangM2000,KovacsR2022,LiLuP2007,Morris1999,Muzychuk1995,Muzychuk1997, Somlai2015}. Recently,  Muzychuk \cite{Muzychukppt} and Xie et al. \cite{XieFX24} found that the dicyclic group $T_{4p}$~($p$ prime) is a DCI-group.

\begin{lemma}\emph{(see \cite{Muzychukppt, XieFX24})}\label{enlemma3}
Let $p$ be a prime. Then $T_{4p}$ is a DCI-group.
\end{lemma}

Lemma \ref{enlemma3} naturally leads to the following result.
\begin{lemma}\label{enlemma4}
Let $p$ be a prime. Then
the two  dicirculant digraphs $\operatorname{Cay}\left(T_{4 p}, S\right)$ and $\operatorname{Cay}\left(T_{4 p}, T\right)$ are isomorphic if and only if there exists some $\alpha \in \operatorname{Aut}\left(T_{4 p}\right)$ such that $\alpha(S)=T$.
\end{lemma}

\begin{lemma}\emph{(see \cite[Proposition 2.3]{Kohl2007})}\label{enlemma5}
For a prime $p$ and the dicyclic group $T_{4p}=\langle a,b\mid a^{2p}=1,a^p=b^2, b^{-1}ab=a^{-1}\rangle$. Then
\begin{equation}\label{enequation3}
\mathrm{Aut(T_{4p})}=\left\{\alpha_{s, t} \mid s \in \mathbb{Z}_{2p}^*, t \in \mathbb{Z}_{2p}\right\},
\end{equation}
where $\alpha_{s, t}\left(a^i\right)=a^{s i}$, $\alpha_{s, t}\left(a^j b\right)=a^{s j+t} b$ for all $i, j \in \mathbb{Z}_{2p}$ and $\mathbb{Z}_{2p}^*$ is the multiplicative group of congruence classes modulo $2p$.
\end{lemma}

\begin{lemma}\emph{(see \cite[Theorem 42]{Shanks1993})}\label{enlemma6}
Let $p$ be a prime and $\mathbb{Z}_{2p}^*$ be the multiplicative group of congruence classes modulo $2p$. Then $\mathbb{Z}_{2p}^*$ is cyclic.

\end{lemma}

\section{Enumerating dicirculant digraphs of order $4p$}

Let
$T_{4p}=\langle a,b\mid a^{2p}=1,a^p=b^2, b^{-1}ab=a^{-1}\rangle$
be the dicyclic group of order $4p$ ($p$ prime).
Take
\begin{equation}\label{enequation2}
A=T_{4 p} \backslash\{e\}=\left\{a^i, a^j b \mid i \in \mathbb{Z}_{2p} \backslash\{0\}, j \in \mathbb{Z}_{2p}\right\}~~\text{and}~~C=\{0,1\}.
\end{equation}
Then Aut $\left(T_{4 p}\right)$ is a permutation group acting on $A$ and $C^A$. For $S \subseteq A$, denote $f_S$ is the characteristic function of $S$, that is, $f_S(a)=1$ if $a \in S$, and $f_S(a)=0$ if $a \in A \backslash S$. Clearly, $f_S \in C^A$ and $C^A$ consists of all characteristic functions on $A$. Note Lemma \ref{enlemma4} that two dicirculant digraphs Cay $\left(T_{4 p}, S\right)$ and $\mathrm{Cay}\left(T_{4 p}, T\right)$ on $T_{4 p}$ are isomorphic if and only if there exists an automorphism $\alpha \in \operatorname{Aut}\left(T_{4 p}\right)$ such that $\alpha(S)=T$, which holds if and only if $f_S, f_T \in C^A$ are $\mathrm{Aut}\left(T_{4 p}\right)$-equivalent. Hence, the number of dicirculant digraphs up to isomorphism is equal to the number of orbits of the group action $\left(\mathrm{Aut}\left(T_{4 p}\right), C^A\right)$. Therefore, we should obtain the cycle index $\mathcal{I}(\mathrm{Aut}\left(T_{4 p}\right), A)$ of the permutation group $\mathrm{Aut}\left(T_{4 p}\right)$ acting $A$ for the purpose of enumerating  dicirculant digraphs.

Set $A_1=\langle a\rangle \backslash\{e\}=\left\{a^i \mid i \in \mathbb{Z}_{2p} \backslash\{0\}\right\}$ and $A_2=\langle a\rangle b=\left\{a^j b \mid j \in \mathbb{Z}_{2p}\right\}$. Then $A=A_1 \cup A_2$. By Lemma \ref{enlemma5},
we have $\alpha_{s, t}\left(A_1\right)=A_1$ and $\alpha_{s, t}\left(A_2\right)=A_2$ for each $\alpha_{s, t} \in \mathrm{Aut}\left(T_{4 p}\right)$.

Let $n$ be a positive integer. The Euler's totient function $\Phi(n)$ is the number of integers $k$ for which $1 \leq k \leq n$ such that the greatest common divisor $\mathrm{gcd}(n, k)$ is equal to 1. Let $n=p_1^{k_1} \cdots p_s^{k_x}$ be the prime factorization of $n$. Then the Euler's product formula states that
$$\Phi(n)=n \prod_{i=1}^r\left(1-\frac{1}{p_i}\right).$$
Let $p\geq 3$. By Lemma \ref{enlemma6}, one can see that $\mathbb{Z}_{2p}^*=\{1,3,\ldots,p-2,p+2,\ldots,2p-1\}$ is a cyclic group of order $\Phi(2p)=p-1$ ($p$ prime).

\begin{lemma}\label{enlemma13}
Let $p$ be an odd prime. Fix an element $1\neq s \in \mathbb{Z}_{2p}^*$. For any $t\in 2\mathbb{Z}_{2p}\setminus\{0\}=\{2,4,6,\ldots,2p-2\}$, there exists an unique $x\in \mathbb{Z}_{2p}^*$ such that
$x-sx=t $ in $2\mathbb{Z}_{2p}$.
\end{lemma}

\begin{proof}
Take an element $1\neq s \in \mathbb{Z}_{2p}^*$. It is clearly that $x-sx\in 2\mathbb{Z}_{2p}\setminus\{0\} ~\text{for~all}~ x,s\in \mathbb{Z}_{2p}^*.$ Assume that there are distinct elements  $x_1,x_2\in \mathbb{Z}_{2p}^*$ such that $x_1-sx_1=x_2-sx_2$ in $2\mathbb{Z}_{2p}\setminus\{0\}$. Then
$$(s-1)(x_2-x_1)\equiv0 \pmod {2p}.$$
This implies that
\begin{equation}\label{enequation10}
o(x_2-x_1)\mid (s-1),
\end{equation}
where $o(x_2-x_1)$ denotes the order of $x_2-x_1$ in $2\mathbb{Z}_{2p}$. Since
$$o(x_2-x_1)=\frac{p}{\gcd(\frac{x_2-x_1}{2},p)}=p,$$
and $s-1\neq p~\text{or}~2p$, (\ref{enequation10}) can not hold. So that is a injection from $\mathbb{Z}_{2p}^*$ to $2\mathbb{Z}_{2p}\setminus\{0\}$. Notice that $|\mathbb{Z}_{2p}^*|=|2\mathbb{Z}_{2p}\setminus\{0\}|=p-1$. The desired result holds.
\qed\end{proof}

Assume that $\mathbb{Z}_{2p}^*=\langle z\rangle$ for some integer $z \in \mathbb{Z}_{2p}^*$. Then, for any $s \in \mathbb{Z}_{2p}^*$, there exists $i_s \in \mathbb{Z}_{p-1}$ such that $s=z^{i_s}$. Furthermore, if $s$ ranges over all elements of $\mathbb{Z}_{2p}^*$, then $i_s$ ranges over all elements of $\mathbb{Z}_{p-1}$. we firstly obtain the cycle index $\mathcal{I}(\mathrm{Aut}\left(T_{4 p}\right), A)$.

\begin{lemma}\label{enlemma7}
Let $p$ be an odd prime. Let $A=T_{4 p} \backslash\{e\}$ and $\alpha_{s, t} \in \mathrm{Aut}\left(T_{4 p}\right)$ be defined as (\ref{enequation2}) and (\ref{enequation3}), respectively. Let $\mathbb{Z}_{2p}^*=\langle z\rangle$. Under the action of $\mathrm{Aut}\left(T_{4 p}\right)$ on $A$, the cycle type of $\alpha_{s, t}$ is given by $\mathcal{T}\left(\alpha_{s, t}\right)=\left(b_1\left(\alpha_{s, t}\right), b_2\left(\alpha_{s, t}\right), \ldots, b_{4 p-1}\left(\alpha_{s, t}\right)\right)$, where
\begin{equation}\label{enequation4}
b_k\left(\alpha_{1, t}\right)=\left\{
\begin{array}{ll}
4 p-1,      &\text{if}~ k=1  ~\text {and}~ t=0,\\[0.2cm]
2p-1,       &\text{if}~ k=1  ~\text {and}~ t\in \mathbb{Z}_{2p} \backslash\{0\},\\[0.2cm]
p,          &\text{if}~ k=2  ~\text {and}~ t=p,\\[0.2cm]
2,          &\text{if}~ k=p  ~\text {and}~ t\in 2\mathbb{Z}_{2p} \backslash\{0\},\\[0.2cm]
1,          &\text{if}~ k=2p ~\text {and}~ t\in (2\mathbb{Z}_{2p}+1) \backslash\{p\},\\[0.2cm]
0,          &\text{otherwise},
\end{array}\right.
\end{equation}
and for each $1 \neq s=z^{i_s} \in \mathbb{Z}_{2p}^*$ (i.e. $i_s\neq 0$) and $t \in 2\mathbb{Z}_{2p}$,
\begin{equation}\label{enequation5}
b_k\left(\alpha_{s, t}\right)=b_k\left(\alpha_{s, 0}\right)=\left\{
\begin{array}{ll}
3,               & \text{if}~ k=1,\\[0.2cm]
4\gcd(i_s, p-1), & \text{if}~ k=\frac{p-1}{\mathrm{gcd}\left(i_s, p-1\right)},\\[0.2cm]
0,               & \text{otherwise},
\end{array}\right.
\end{equation}
and for each $1 \neq s=z^{i_s} \in \mathbb{Z}_{2p}^*$ (i.e. $i_s\neq 0$) and $t \in 2\mathbb{Z}_{2p}+1$,
\begin{equation}\label{enequation11}
b_k\left(\alpha_{s, t}\right)=b_k\left(\alpha_{s, 1}\right)=\left\{
\begin{array}{ll}
1,               & \text{if}~ k=1,\\[0.2cm]
1,               & \text{if}~ k=2,\\[0.2cm]
4\gcd(i_s, p-1), & \text{if}~ k=\frac{p-1}{\mathrm{gcd}\left(i_s, p-1\right)},~\text{and}~ \frac{p-1}{\mathrm{gcd}\left(i_s, p-1\right)} ~\text{is~even},\\[0.2cm]
2\gcd(i_s, p-1), & \text{if}~ k=\frac{p-1}{\mathrm{gcd}\left(i_s, p-1\right)},~\text{and}~ \frac{p-1}{\mathrm{gcd}\left(i_s, p-1\right)} ~\text{is~odd},\\[0.2cm]
\gcd(i_s, p-1), & \text{if}~ k=\frac{2p-2}{\mathrm{gcd}\left(i_s, p-1\right)},~\text{and}~ \frac{p-1}{\mathrm{gcd}\left(i_s, p-1\right)} ~\text{is~odd},\\[0.2cm]
0,               & \text{otherwise}.
\end{array}\right.
\end{equation}
\end{lemma}

\begin{proof}
Let $A_1$ and $A_2$ be defined as above. Since $\alpha_{s, t}\left(A_1\right)=A_1$ and $\alpha_{s, t}\left(A_2\right)=A_2$ for each $\alpha_{s, t} \in \mathrm{Aut}\left(T_{4 p}\right)$, we have $b_{2p+1}\left(\alpha_{s, t}\right)=\cdots=b_{4 p-1}\left(\alpha_{s, t}\right)=0$. For $\alpha_{s, t} \in \mathrm{Aut}\left(T_{4 p}\right)$, we consider it in the following two cases.

\noindent\emph{Case 1.} $s=1$.
Note that $\alpha_{1, t}\left(a^i\right)=a^i$ for each $i \in \mathbb{Z}_{2p} \backslash\{0\}$. Then the permutation $\alpha_{1, t}$ splits $A_1$ into $2p-1$ cycles  of length $1$. Notice that $\alpha_{1, t}\left(a^j b\right)=a^{j+t} b$ for $j \in \mathbb{Z}_{2p}$. If $t=0$, then $\alpha_{1,0}\left(a^j b\right)=a^j b$ for each $j \in \mathbb{Z}_{2p}$. Therefore, $\alpha_{1, t}$ splits $A_2$ into $2p$ cycles of length $1$. If $t \in \mathbb{Z}_{2p} \backslash\{0\}$, then the order of $t$ is
\begin{equation*}
o(t)=\frac{2p}{\mathrm{gcd}(t, 2p)}=\left\{
\begin{array}{ll}
2,     &t=p,\\[0.2cm]
p,     &t\in 2\mathbb{Z}_{2p} \backslash\{0\},\\[0.2cm]
2p,    &t\in (2\mathbb{Z}_{2p}+1) \backslash\{p\}.
\end{array}\right.
\end{equation*}

If $t=p$, then $a^j b \in A_2$ is in the cycle $\left(a^j b, \alpha_{1,p}(a^{j} b)\right)=\left(a^j b, a^{j+p} b\right)$ for any $j \in \mathbb{Z}_{2p}$. Therefore, the permutation $\alpha_{1, p}$ splits $A_2$ into $p$ cycles of length $2$.
If $t\in 2\mathbb{Z}_{2p} \backslash\{0\}$, then $a^j b \in A_2$ is in the cycle
$$\left(a^j b, \alpha_{1,t}(a^{j} b),\ldots, \alpha_{1,t}^{p-1}(a^{j} b) \right)=\left(a^j b, a^{j+t} b,\ldots, a^{j+(p-1)t} b\right),$$
for any $j \in \mathbb{Z}_{2p}$. Therefore, the permutation $\alpha_{1, t}$ splits $A_2$ into $2$ cycles of length $p$.
If $t\in (2\mathbb{Z}_{2p}+1) \backslash\{p\}$, then $a^j b \in A_2$ is in the cycle
$$\left(a^j b, \alpha_{1,t}(a^{j} b),\ldots, \alpha_{1,t}^{2p-1}(a^{j} b) \right)=\left(a^j b, a^{j+t} b,\ldots, a^{j+(2p-1)t} b\right),$$
for any $j \in \mathbb{Z}_{2p}$. Therefore, the permutation $\alpha_{1, t}$ splits $A_2$ into one cycle of length $2p$. Therefore, we have got the cycle type of $\alpha_{1, t}$, as shown in (\ref{enequation4}).

\noindent\emph{Case 2.} $s\neq 1$, say $s=z^{i_s}~(i_s\neq 0)$.
Note that $\alpha_{s, t}\left(a^i\right)=a^{s i}$ for $i \in \mathbb{Z}_{2p} \backslash\{0\}$ and $\alpha_{s, t}\left(a^j b\right)=a^{s j+t} b$ for $j \in \mathbb{Z}_{2p}$.
Firstly, we claim that $\alpha_{s, t}$  has the same cycle type as $\alpha_{s, 0}$, and $\alpha_{s, t+1}$  has the same cycle type as $\alpha_{s, 1}$, for each $t \in 2\mathbb{Z}_{2p}\setminus\{0\}$.

Since
$$\alpha_{s, t}\left(a^i\right)=\alpha_{s, t+1}\left(a^i\right)=a^{s i}=\alpha_{s, 0}\left(a^i\right)=\alpha_{s, 1}\left(a^i\right),$$ $\alpha_{s, t}$,  $\alpha_{s, t+1}$, $\alpha_{s, 0}$ and $\alpha_{s, 1}$ have the same cycle type in $A_1$ for each $t \in 2\mathbb{Z}_{2p}\setminus\{0\}$.

Set $t \in 2\mathbb{Z}_{2p}\setminus\{0\}$. Notice Lemma \ref{enlemma13} that there exists an unique $x\in \mathbb{Z}_{2p}^*$ such that
$x-sx=t $. Define $\beta$ is a bijection in $A_2$ such that $\beta(a^j b)=a^{j+x} b$ for $j \in \mathbb{Z}_{2p}$.
Assume that $\left(a^{j_0} b, a^{j_1} b, \ldots, a^{j_{r-1}} b\right)$ ($j_0, j_1, \ldots, j_{r-1} \in \mathbb{Z}_{2p}$) is a cycle of $\alpha_{s, 0}$, i.e. $j_{l}=s j_{l-1}$ for $l \in \mathbb{Z}_{r}$. Then, for each $l \in \mathbb{Z}_r$, we have
$\beta\left(a^{j_{l}} b\right)=a^{j_{l}+x}b, $
and
\begin{align*}
\alpha_{s, t}\left(\beta\left(a^{j_{l-1}} b\right)\right)&=\alpha_{s, t}\left(a^{j_{l-1}+x} b\right)
=a^{s\left(j_{l-1}+x\right)+t} b
=a^{j_{l}+sx+t} b=a^{j_{l}+x} b
=\beta\left(a^{j_{l}}b\right).
\end{align*}
Thus
$\left(\beta\left(a^{j_0} b\right), \beta\left(a^{j_1} b\right), \ldots, \beta\left(a^{j_{r-1}} b\right)\right)$
is a cycle of $\alpha_{s, t}$.

Similarly, we get that if  $\left(a^{j_0} b, a^{j_1} b, \ldots, a^{j_{m-1}} b\right)$ ($j_0, j_1, \ldots, j_{m-1} \in \mathbb{Z}_{2p}$) is a cycle of $\alpha_{s, 1}$, then
$$\left(\beta\left(a^{j_0} b\right), \beta\left(a^{j_1} b\right), \ldots, \beta\left(a^{j_{r-1}} b\right)\right)$$
is a cycle of $\alpha_{s, t+1}$.
Therefore, $\alpha_{s, t}$ and $\alpha_{s, 0}$ have the same cycle type in $A_2$, and $\alpha_{s, t+1}$  has the same cycle type as $\alpha_{s, 1}$ in $A_2$ for each $t \in 2\mathbb{Z}_{2p}\setminus\{0\}$.  Hence, we only need to study the cycle type of $\alpha_{s, 0}$ in $A=A_1 \cup A_2$ and the cycle type of $\alpha_{s, 1}$ in $A_2$.

\noindent\emph{Case 2.1.} We firstly study the cycle type of $\alpha_{s, 0}$ in $A=A_1 \cup A_2$.
For an element $s\in\mathbb{Z}_{2p}^*$, we have the order of $s$ is
$$o(s)=o\left(z^{i_s}\right)=\frac{p-1}{\mathrm{gcd}\left(i_s, p-1\right)}.$$
Then, $a^i \in A_1$ is in the cycle
$$\left(a^i, \alpha_{s, 0}\left(a^i\right), \alpha_{s, 0}^2\left(a^i\right), \ldots, \alpha_{s, 0}^{o(s)-1}\left(a^i\right)\right)=\left(a^i, a^{s i}, a^{s^2 i}, \ldots, a^{s^{o(s)-1} i}\right),$$
and
$a^j b\in A_2$ is in the cycle
$$\left(a^j b, \alpha_{s, 0}\left(a^j b\right), \alpha_{s, 0}^2\left(a^j b\right), \ldots, \alpha_{s, 0}^{o(s)-1}\left(a^j b\right)\right)=\left(a^j b, a^{s j} b, a^{s^2j} b, \ldots, a^{s^{o(s)-1}j}b\right),$$
for any $i, j \in \mathbb{Z}_{2p} \backslash\{0,p\}$.
Note that $\alpha_{s, 0}\left(a^p\right)=a^{sp}=a^{p}$ as $s\in \mathbb{Z}_{2p}^*$ is an odd number. So $a^p \in A_1$ is in the cycle $\left(a^p \right)$.  Note also that $\alpha_{s, 0}\left(a^0 b\right)=a^0 b$ and $\alpha_{s, 0}\left(a^p b\right)=a^{sp} b=a^{p} b$. So $a^0 b \in A_2$ is in the cycle $\left(a^0 b\right)$ and $a^p b \in A_2$ is in the cycle $\left(a^p b\right)$.
Thus, the permutation $\alpha_{s, 0}$ splits $A_1$ into
$\frac{2p-2}{o(s)}=2\mathrm{gcd}\left(i_s, p-1\right)$
cycles of length $o(s)$ and one cycle of length 1, and splits $A_2$ into $\frac{2p-2}{o(s)}=2\mathrm{gcd}\left(i_s, p-1\right)$
cycles of length $o(s)$ and two cycles of length 1.
Therefore, we have got the cycle type of $\alpha_{s, t}$, as shown in (\ref{enequation5}).

\noindent\emph{Case 2.2.} We next study the cycle type of $\alpha_{s, 1}$ in $A_2$. Clearly, $a^jb\in A_2$ is in the cycle
$$\left(a^j b, \alpha_{s, 1}\left(a^j b\right), \alpha_{s, 1}^2\left(a^j b\right), \ldots, \alpha_{s, 1}^{m-1}\left(a^j b\right)\right)=\left(a^j b, a^{s j+1} b, a^{s^2j+s+1} b, \ldots, a^{s^{m-1}j+s^{m-2}+\cdots+1}b\right),$$
where $m$ is the least number such that
$$s^{m}j+s^{m-1}+s^{m-2}+\cdots+1 \equiv j \pmod {2p},$$
that is,
$$(s^{m-1}+s^{m-2}+\cdots+1)\left((s-1)j+1\right) \equiv 0 \pmod {2p}.$$
If $(s-1)j+1=p$, then $m=2$ and $j\neq p$. Therefore, $a^jb$ is in the cycle $(a^jb, a^{sj+1}b)$. By Lemma \ref{enlemma13}, such $j$ and cycle is unique for each $s$.
If $(s-1)j+1\neq p$, then $o((s-1)j+1)=2p$ in $\mathbb{Z}_{2p}$. Hence, $$s^{m-1}+s^{m-2}+\cdots+1\equiv 0 \pmod {2p},$$
Notice that the least of $y$ such that
$$s^y-1=(s-1)(s^{y-1}+s^{y-2}+\cdots+1)\equiv 0 \pmod {2p}$$
is $o(s)$ in $\mathbb{Z}_{2p}^*$. Then $m\geq o(s)$. Since $s-1\neq p,2p$, we have
$$p\mid (s^{o(s)-1}+s^{o(s)-1}+\cdots+1).$$
If $o(s)$ is even, then $s^{o(s)-1}+s^{o(s)-1}+\cdots+1$ is even. Hence $m=o(s)$. Thus, the permutation $\alpha_{s, 1}$ splits $A_2$ into $\frac{2p-2}{o(s)}=2\mathrm{gcd}\left(i_s, p-1\right)$
cycles of length $o(s)$ and one cycle of length 2.
If $o(s)$ is odd, then  $s^{o(s)-1}+s^{o(s)-2}+\cdots+1=p$. Therefore,
$$2(s^{o(s)-1}+s^{o(s)-2}+\cdots+1)=s^{2o(s)-1}+s^{2o(s)-2}+\cdots+1=2p.$$
Then $m=2o(s)$. Thus, the permutation $\alpha_{s, 1}$
splits $A_2$ into
$\frac{2p-2}{2o(s)}=\mathrm{gcd}\left(i_s, p-1\right)$
cycles of length $2o(s)$ and one cycle of length 2.
Therefore, we have got the cycle type of $\alpha_{s, t}$, as shown in (\ref{enequation11}).
\qed\end{proof}

According to Lemma \ref{enlemma7}, we obtain the cycle index of $\mathrm{Aut}\left(T_{4p}\right)$ acting on $A=T_{4p} \backslash\{e\}$.

\begin{lemma}\label{enlemma9}
Let $p$ be an odd prime. The cycle index of $\mathrm{Aut}\left(T_{4 p}\right)$ acting on $A=T_{4 p} \backslash\{e\}$ is given by
\begin{align*}
\mathcal{I}\left(\mathrm{Aut}\left(T_{4 p}\right), A\right)&=\frac{1}{2p} x_1^{2p-1}\left(x_p^2+x_{2p}-x_1^{2p}-x_2^p\right)
+\frac{1}{2(p-1)} x_1^3 \cdot \sum_{d \mid(p-1)} \Phi(d) x_d^{\frac{4(p-1)}{d}}\\[0.3cm]
&~~+\frac{1}{2(p-1)}x_1x_2\sum_{\substack{d\mid (p-1)\\ d~\text{even}}}\Phi(d)x_d^{\frac{4(p-1)}{d}}+\frac{1}{2(p-1)}x_1x_2\sum_{\substack{d\mid (p-1)\\ d~\text{odd}}}\Phi(d)x_d^{\frac{2(p-1)}{d}}x_{2d}^{\frac{p-1}{d}},
\end{align*}
where $\Phi(\cdot)$ denotes the Euler's totient function.
\end{lemma}

\begin{proof}
By Lemma \ref{enlemma7}, the cycle index of $\mathrm{Aut}\left(T_{4 p}\right)$ acting on $A=T_{4p} \backslash\{e\}$ is
\begin{align*}
&\mathcal{I}\left(\mathrm{Aut}(T_{4 p}), A\right)\\
&=\frac{1}{\left|\mathrm{Aut}\left(T_{4 p}\right)\right|} \sum_{\alpha_{s, t} \in \mathrm{Aut}\left(T_{4p}\right)} x_1^{b_1\left(\alpha_{s, t}\right)} x_2^{b_2\left(\alpha_{s, t}\right)} \cdots x_{4p-1}^{b_{4p-1}\left(\alpha_{s, t}\right)}\\[0.3cm]
&=\frac{1}{2p(p-1)} \sum_{s\in\mathbb{Z}_{2p}^*}\sum_{t\in\mathbb{Z}_{2p}} x_1^{b_1\left(\alpha_{s, t}\right)} x_2^{b_2\left(\alpha_{s, t}\right)} \cdots x_{4p-1}^{b_{4p-1}\left(\alpha_{s, t}\right)}\\[0.3cm]
&=\frac{1}{2p(p-1)} \left[\sum_{t\in\mathbb{Z}_{2p}} x_1^{b_1\left(\alpha_{1, t}\right)} x_2^{b_2\left(\alpha_{1, t}\right)} \cdots x_{4p-1}^{b_{4p-1}\left(\alpha_{1, t}\right)}\right.\\[0.3cm]
&\left.~~~~+\sum_{s\in\mathbb{Z}_{2p}^*\backslash\{1\}}\left(\sum_{t\in2\mathbb{Z}_{2p}} x_1^{b_1\left(\alpha_{s, t}\right)} x_2^{b_2\left(\alpha_{s, t}\right)} \cdots x_{4p-1}^{b_{4p-1}\left(\alpha_{s, t}\right)}
+\sum_{t\in2\mathbb{Z}_{2p}+1} x_1^{b_1\left(\alpha_{s, t}\right)} x_2^{b_2\left(\alpha_{s, t}\right)} \cdots x_{4p-1}^{b_{4p-1}\left(\alpha_{s, t}\right)}\right)\right]\\[0.3cm]
&=\frac{1}{2p(p-1)} \left[\sum_{t\in\mathbb{Z}_{2p}} x_1^{b_1\left(\alpha_{1, t}\right)} x_2^{b_2\left(\alpha_{1, t}\right)} \cdots x_{4p-1}^{b_{4p-1}\left(\alpha_{1, t}\right)}+p\sum_{s\in\mathbb{Z}_{2p}^*\backslash\{1\}} x_1^{b_1\left(\alpha_{s, 0}\right)} x_2^{b_2\left(\alpha_{s, 0}\right)} \cdots x_{4p-1}^{b_{4p-1}\left(\alpha_{s, 0}\right)}\right.\\[0.3cm]
&\left.~~~~+p\sum_{s\in\mathbb{Z}_{2p}^*\backslash\{1\}} x_1^{b_1\left(\alpha_{s, 1}\right)} x_2^{b_2\left(\alpha_{s, 1}\right)} \cdots x_{4p-1}^{b_{4p-1}\left(\alpha_{s, 1}\right)}\right]\\[0.3cm]
&=\frac{1}{2p(p-1)} \left[x_1^{4p-1}+x_1^{2p-1}\left(x_2^p+(p-1)x_p^2+(p-1)x_{2p}\right)+px_1^3\sum_{s=z^{i_s}\in \mathbb{Z}_{2p}^*\backslash\{1\}}x_{\frac{p-1}{\gcd(i_s, p-1)}}^{4\gcd(i_s,p-1)}\right.\\[0.3cm]
&\left.~~~~+px_1x_2\left(\sum_{\substack{s=z^{i_s}\in \mathbb{Z}_{2p}^*\backslash\{1\}\\\frac{p-1}{\gcd(i_s, p-1)}~\text{even}}}x_{\frac{p-1}{\gcd(i_s, p-1)}}^{4\gcd(i_s,p-1)}+\sum_{\substack{s=z^{i_s}\in \mathbb{Z}_{2p}^*\backslash\{1\}\\\frac{p-1}{\gcd(i_s, p-1)}~\text{odd}}}x_{\frac{p-1}{\gcd(i_s, p-1)}}^{2\gcd(i_s,p-1)}x_{\frac{2p-2}{\gcd(i_s, p-1)}}^{\gcd(i_s,p-1)}\right)\right]\\[0.3cm]
&=\frac{1}{2p(p-1)} \left[x_1^{4p-1}+x_1^{2p-1}\left(x_2^p+(p-1)x_p^2+(p-1)x_{2p}\right)+px_1^3\sum_{i_s\in \mathbb{Z}_{p-1}\setminus\{0\}}x_{\frac{p-1}{\gcd(i_s, p-1)}}^{4\gcd(i_s,p-1)}\right.\\[0.3cm]
&\left.~~~~+px_1x_2\left(\sum_{\substack{i_s\in \mathbb{Z}_{p-1}\backslash\{0\}\\\frac{p-1}{\gcd(i_s, p-1)}~\text{even}}}x_{\frac{p-1}{\gcd(i_s, p-1)}}^{4\gcd(i_s,p-1)}+\sum_{\substack{i_s\in \mathbb{Z}_{p-1}\backslash\{0\}\\\frac{p-1}{\gcd(i_s, p-1)}~\text{odd}}}x_{\frac{p-1}{\gcd(i_s, p-1)}}^{2\gcd(i_s,p-1)}x_{\frac{2p-2}{\gcd(i_s, p-1)}}^{\gcd(i_s,p-1)}\right)\right]\\[0.3cm]
&=\frac{1}{2p(p-1)} \left(x_1^{4p-1}+x_1^{2p-1}\left(x_2^p+(p-1)x_p^2+(p-1)x_{2p}\right)+px_1^3\sum_{\substack{d\mid (p-1)\\d\neq 1}}\Phi(d)x_d^{\frac{4(p-1)}{d}}\right.\\[0.3cm]
&\left.~~~~+px_1x_2\sum_{\substack{d\mid (p-1)\\ d~\text{even}}}\Phi(d)x_d^{\frac{4(p-1)}{d}}+px_1x_2\sum_{\substack{d\mid (p-1)\\d\neq 1, d~\text{odd}}}\Phi(d)x_d^{\frac{2(p-1)}{d}}x_{2d}^{\frac{p-1}{d}}\right)\\[0.3cm]
&=\frac{1}{2p(p-1)} \left(x_1^{4p-1}+x_1^{2p-1}\left(x_2^p+(p-1)x_p^2+(p-1)x_{2p}\right)+px_1^3\sum_{d\mid (p-1)}\Phi(d)x_d^{\frac{4(p-1)}{d}}\right.\\[0.3cm]
&\left.~~~~-px_1^{4p-1}+px_1x_2\sum_{\substack{d\mid (p-1)\\ d~\text{even}}}\Phi(d)x_d^{\frac{4(p-1)}{d}}+px_1x_2\sum_{\substack{d\mid (p-1)\\ d~\text{odd}}}\Phi(d)x_d^{\frac{2(p-1)}{d}}x_{2d}^{\frac{p-1}{d}}-px_1^{2p-1}x_2^p\right)\\[0.3cm]
&=\frac{1}{2p} x_1^{2p-1}\left(x_p^2+x_{2p}-x_1^{2p}-x_2^p\right)
+\frac{1}{2(p-1)} x_1^3 \cdot \sum_{d \mid(p-1)} \Phi(d) x_d^{\frac{4(p-1)}{d}}\\[0.3cm]
&~~~~+\frac{1}{2(p-1)}x_1x_2\sum_{\substack{d\mid (p-1)\\ d~\text{even}}}\Phi(d)x_d^{\frac{4(p-1)}{d}}+\frac{1}{2(p-1)}x_1x_2\sum_{\substack{d\mid (p-1)\\ d~\text{odd}}}\Phi(d)x_d^{\frac{2(p-1)}{d}}x_{2d}^{\frac{p-1}{d}},
\end{align*}
where $\Phi(\cdot)$ denotes the Euler's totient function.
\qed\end{proof}

By Lemmas \ref{enlemma8} and  \ref{enlemma9}, we get the number of dicirculant digraphs up to isomorphism immediately.
\begin{theorem}\label{entheorem1}
Let $p$ be an odd prime. The number of dicirculant digraphs up to isomorphism is equal to
\begin{align}\label{enequation6}
\nonumber\mathcal{N}&=\frac{2^{2p-1}}{p}\left(3-2^{2p-1}-2^{p-1}\right)+\frac{4}{p-1} \cdot \sum_{d \mid(p-1)} \Phi(d) 2^{\frac{4(p-1)}{d}}\\[0.3cm]
&~~+\frac{2}{(p-1)}\sum_{\substack{d\mid (p-1)\\ d~\text{even}}}\Phi(d)2^{\frac{4(p-1)}{d}}+\frac{2}{(p-1)}\sum_{\substack{d\mid (p-1)\\ d~\text{odd}}}\Phi(d)2^{\frac{3(p-1)}{d}},
\end{align}
where $\Phi(\cdot)$ is the Euler's totient function.
\end{theorem}

In \cite{Mishna2000}, Mishna calculated the number of the circulant digraphs of order $p$~($p$ prime) up to isomorphism. With the similar proof of Theorem 2.12 in \cite{Mishna2000}, we get the number of circulant digraphs of order $2p$~($p$ prime).
\begin{lemma}\label{enlemma11}
Let $p$ be an odd prime. The number of circulant digraphs  of order $2p$ up to isomorphism is given by
\begin{equation}\label{enequation7}
\mathcal{N}_c=\frac{1}{p-1} \sum_{d \mid(p-1)} \Phi(d) 2^{\frac{2p-2}{d}+1},
\end{equation}
where $\Phi(\cdot)$ is the Euler's totient function.
\end{lemma}

Recall that a Cayley digraph $\operatorname{Cay}(G, S)$ is connected if and only if $\langle S\rangle=G$. Thus, for $S \subseteq A=T_{4 p} \backslash\{e\}$, the dicirculant digraph $\Cay\left(T_{4 p}, S\right)$ is disconnected if and only if
\begin{align*}
S \subseteq \langle a\rangle \backslash\{e\} ~~\text{or}~~ S=\left\{a^j b\right\}~~\text{or}~~ S=\left\{a^j b, a^{p+j}b\right\}\\[0.2cm]
 \text{or}~~ S=\left\{a^p, a^j b\right\} ~~\text{or}~~ S=\left\{a^p, a^j b, a^{p+j}b\right\} ~~\text{for}~~ j \in \mathbb{Z}_{2p},
\end{align*}
as $p$ is a prime. Note that
\begin{align*}
\Cay\left(T_{4 p},\left\{a^j b, a^{p+j}b\right\}\right) \cong \Cay\left(T_{4 p},\{b, a^pb\}\right),
\Cay\left(T_{4 p},\left\{a^p, a^jb\right\}\right) \cong \Cay\left(T_{4 p},\{a^p, b\}\right),\\[0.2cm]
\Cay\left(T_{4 p},\left\{a^j b\right\}\right) \cong \Cay\left(T_{4 p},\{b\}\right),
\Cay\left(T_{4 p},\left\{a^p, a^j b, a^{p+j}b\right\}\right) \cong \Cay\left(T_{4 p},\{a^p, b, a^pb\}\right),
\end{align*}
for each $j$ since $\alpha_{1, j}(b)=a^j b$ and $\alpha_{1, j}(a^pb)=a^{p+j} b$. Hence, from Theorem \ref{entheorem1} and Lemma \ref{enlemma11}, we have the number of connected dicirculant digraphs immediately.

\begin{theorem}
Let $p$ be an odd prime. The number of connected dicirculant digraphs up to isomorphism is equal to
$$
\mathcal{N}'=\mathcal{N}-\mathcal{N}_c-4,
$$
where $\mathcal{N}$ and $\mathcal{N}_c$ are shown in (\ref{enequation6}) and (\ref{enequation7}), respectively.
\end{theorem}

\begin{theorem}\label{entheorem2}
Let $p$ be an odd prime. The number of dicirculant digraphs of out-degree $k$ up to isomorphism is $1$ if $k=0$ or $k=4p-1$, and  $\mathcal{M}_k$ if $1\leq k\leq 4p-2$,
where
\begin{align}\label{enequation8}\nonumber
\mathcal{M}_k&=\frac{1}{2p}\left[\sum_{j=0}^2\binom{2}{j}\binom{2p-1}{k-pj}+\binom{2p-1}{k}+\binom{2p-1}{k-2p}
-\binom{4p-1}{k}-\sum_{j=0}^p\binom{p}{j}\binom{2p-1}{k-2j}\right]\\[0.3cm]\nonumber
&~~+\frac{1}{(p-1)}\left[\sum_{\substack{d \mid\gcd(p-1,k)\\d~\text{even}}}\Phi(d)\binom{\frac{4(p-1)}{d}}{\frac{k}{d}}+2\sum_{\substack{d \mid\gcd(p-1,k-1)\\d~\text{even}}}\Phi(d)\binom{\frac{4(p-1)}{d}}{\frac{k-1}{d}}\right.\\[0.3cm]\nonumber
&~~~~\left.+2\sum_{\substack{d \mid\gcd(p-1,k-2)\\d~\text{even}}}\Phi(d)\binom{\frac{4(p-1)}{d}}{\frac{k-2}{d}}+\sum_{\substack{d \mid\gcd(p-1,k-3)\\d~\text{even}}}\Phi(d)\binom{\frac{4(p-1)}{d}}{\frac{k-3}{d}}\right]\\[0.3cm]\nonumber
&~~+\frac{1}{2(p-1)}\left[\sum_{i=0,3}\sum_{\substack{d \mid\gcd(p-1,k-i)\\d~\text{odd}}}\left(\Phi(d)\binom{\frac{4(p-1)}{d}}{\frac{k-i}{d}}
+\sum_{j=0}^{\frac{p-1}{d}}\binom{\frac{2(p-1)}{d}}{\frac{k-i-2dj}{d}}\binom{\frac{p-1}{d}}{j}\right)\right.\\[0.3cm]
&~~~~+\left.\sum_{i=1,2}\sum_{\substack{d \mid\gcd(p-1,k-i)\\d~\text{odd}}}\left(3\Phi(d)\binom{\frac{4(p-1)}{d}}{\frac{k-i}{d}}
+\sum_{j=0}^{\frac{p-1}{d}}\binom{\frac{2(p-1)}{d}}{\frac{k-i-2dj}{d}}\binom{\frac{p-1}{d}}{j}\right)\right].
\end{align}
\end{theorem}

\begin{proof}
Clearly, if $k=0$ (respectively, $k=4 p-1$ ), then there exists only one dicirculant digraph $\Cay\left(T_{4 p}, \emptyset\right)$ (respectively, $\left.\Cay\left(T_{4 p}, T_{4 p} \backslash\{e\}\right)\right)$ of out-degree $k$. We just need to consider $1 \leq k \leq 4 p-2$. By Lemma \ref{enlemma9}, the cycle index of $\mathrm{Aut}\left(T_{4 p}\right)$ acting on $A=T_{4 p} \setminus\{e\}$ is
\begin{align*}
\mathcal{I}\left(\mathrm{Aut}\left(T_{4 p}\right), A\right)&=\frac{1}{2p} x_1^{2p-1}\left(x_p^2+x_{2p}-x_1^{2p}-x_2^p\right)
+\frac{1}{2(p-1)} x_1^3 \cdot \sum_{d \mid(p-1)} \Phi(d) x_d^{\frac{4(p-1)}{d}}\\[0.3cm]
&~~+\frac{1}{2(p-1)}x_1x_2\sum_{\substack{d\mid (p-1)\\ d~\text{even}}}\Phi(d)x_d^{\frac{4(p-1)}{d}}+\frac{1}{2(p-1)}x_1x_2\sum_{\substack{d\mid (p-1)\\ d~\text{odd}}}\Phi(d)x_d^{\frac{2(p-1)}{d}}x_{2d}^{\frac{p-1}{d}}.
\end{align*}
Putting $x_i=1+x^i$ in the above equation, we have
\begin{align*}
Q(x)&=\frac{1}{2p} (1+x)^{2p-1}\left[(1+x^p)^2+1+x^{2p}-(1+x)^{2p}-(1+x^2)^p\right]\\[0.3cm]
&~~+\frac{1}{2(p-1)} (1+x)^3 \cdot \sum_{d \mid(p-1)} \Phi(d) (1+x^d)^{\frac{4(p-1)}{d}}\\[0.3cm]
&~~+\frac{1}{2(p-1)}(1+x)(1+x^2)\sum_{\substack{d\mid (p-1)\\ d~\text{even}}}\Phi(d)(1+x^d)^{\frac{4(p-1)}{d}}\\[0.3cm]
&~~+\frac{1}{2(p-1)}(1+x)(1+x^2)\sum_{\substack{d\mid (p-1)\\ d~\text{odd}}}\Phi(d)(1+x^d)^{\frac{2(p-1)}{d}}(1+x^{2d})^{\frac{p-1}{d}}.
\end{align*}
Recall that the two dicirculant digraphs $\Cay\left(T_{4 p}, S\right)$ and $\Cay\left(T_{4 p}, T\right)$ are isomorphic if and only if $S$ and $T$ are $\mathrm{Aut}\left(T_{4 p}\right)$-equivalent. Thus the number of dicirculant digraphs of out-degree $k$ up to isomorphism is equal to the number of $\mathrm{Aut}\left(T_{4 p}\right)$-equivalent $k$-subsets of $A$. By Lemma \ref{enlemma10}, it's equal to the coefficient of $x^k$ in the polynomial $Q(x)$. We let $\binom{n}{j}=0$ if $j<0$ or $j>n$. Then the coefficient of $x^k$ in the polynomial $Q(x)$ is clearly given in (\ref{enequation8}).
\qed\end{proof}

In \cite{Mishna2000}, Mishna calculated the number of the circulant digraphs of order $p$~($p$ prime) and out degree $k$  up to isomorphism. With the similar proof of Section 2.4 in \cite{Mishna2000}, we give the number of circulant digraphs of order $2p$~($p$ prime) and out degree $k~(0\leq k\leq 2p-1)$.
\begin{lemma}\label{enlemma12}
Let $p$ be an odd prime. The number of circulant digraphs of out degree $k$ up to isomorphism is equal to
\begin{equation}\label{enequation9}
\mathcal{M}_{c, k}=\frac{1}{p-1}\left[\sum_{d \mid \gcd(p-1, k)} \Phi(d)\binom{\frac{2(p-1)}{d}}{\frac{k}{d}}+\sum_{d \mid \gcd(p-1, k-1)} \Phi(d)\binom{\frac{2(p-1)}{d}}{\frac{k-1}{d}}\right],
\end{equation}
where $0\leq k\leq 2p-1$ and $\Phi(\cdot)$ is the Euler's totient function.
\end{lemma}

Let $\Cay\left(T_{4 p}, S\right)$ be a dicirculant digraph with $|S|=k$. If $k=0$ or $1$ , then $\Cay\left(T_{4 p}, S\right)$ is obviously disconnected. If $k=2$, then $\Cay\left(T_{4 p}, S\right)$ is disconnected if and only if
\begin{align*}
S \subseteq \langle a\rangle \backslash\{e\} ~~\text{or}~~ S=\left\{a^j b, a^{p+j}b\right\}~~\text{or}~~ S=\left\{a^p, a^j b\right\}, ~~\text{for}~~ j \in \mathbb{Z}_{2p}.
\end{align*}
If $k=3$, then $\Cay\left(T_{4 p}, S\right)$ is disconnected if and only if
\begin{align*}
S \subseteq \langle a\rangle \backslash\{e\} ~~\text{or}~ S=\left\{a^p, a^j b, a^{p+j}b\right\} ~~\text{for}~~ j \in \mathbb{Z}_{2p}.
\end{align*}
For $4 \leq k \leq 2p-1$, $\Cay\left(T_{4 p}, S\right)$ is disconnected if and only if $S \subseteq \langle a\rangle \backslash\{e\}$, and for $2p \leq k \leq 4 p-1, \Cay\left(T_{4 p}, S\right)$ must be connected. By Theorem \ref{entheorem2} and Lemma \ref{enlemma11}, we get the number of connected dicirculant digraphs of order $4p$ and out-degree $k$.

\begin{theorem}
Let $p$ be an odd prime. The number of connected dicirculant digraphs of order $4p$ and out-degree $k$ up to isomorphism is
\begin{equation*}
\mathcal{M'}_k=\left\{
\begin{array}{ll}
0,                                         & \text { if~}~k=0,1, \\[0.2cm]
\mathcal{M}_2-\mathcal{M}_{c, 2}-2,        & \text { if }~ k=2,   \\[0.2cm]
\mathcal{M}_3-\mathcal{M}_{c, 3}-1,        & \text { if }~ k=3,   \\[0.2cm]
\mathcal{M}_k-\mathcal{M}_{c, k},          & \text { if }~ 4 \leq k \leq 2p-1, \\[0.2cm]
\mathcal{M}_k,                             & \text { if }~ 2p \leq k \leq 4 p-1.
\end{array}\right.
\end{equation*}
where $\mathcal{M}_k$ and $\mathcal{M}_{c, k}$ are shown in (\ref{enequation8}) and (\ref{enequation9}), respectively.
\end{theorem}

\begin{theorem}
There are $36$ dicirculant digraphs of order $8$ up to isomorphism, in which $26$ are connected.
\end{theorem}
\begin{proof}
Note that $T_8=\langle a, b\mid a^4=1, a^2=b^2, b^{-1}ab=a^{-1}\rangle=\left\{a^i, a^j b \mid 0 \leq i, j \leq 3\right\}$ is the quaternion group. Let $A=T_8 \backslash\{1\}$. Then all the representative elements of $\operatorname{Aut}\left(T_8\right)$-equivalent classes of subsets of $A$ are as follows:
\begin{align*}
&\emptyset, \{a\}, \{a^2\}, \{b\}, \{a,a^2\}, \{a,a^3\}, \{a,b\},
\{a^2,b\}, \{b,ab\},\{b,a^2b\}, \{a,a^2,a^3\},\{a,a^2,b\},\\&\{a,a^3,b\}, \{a,b,ab\}, \{a^2,b,ab\}, \{b,ab,a^2b\}, \{a,b,a^2b\}, \{a^2,b,a^2b\},
\{a,a^2,a^3,b\},\{a,a^2,b,ab\},\\&\{a,a^2,b,a^2b\},
\{a,a^3,b,ab\}, \{a,a^3,b,a^2b\}, \{a,b,ab,a^2b\},\{b,ab,a^2b,a^3b\},\{a^2,b,ab,a^2b\}\\
&\{a,a^2,a^3,b,ab\},
\{a,a^2,a^3,b,a^2b\},\{a,a^2,b,ab,a^2b\},
\{a,a^3,b,ab,a^2b\},\{a,b,ab,a^2b,a^3b\},\\&\{a^2,b,ab,a^2b,a^3b\},
\{a,a^2,a^3,b,ab,a^2b\},\{a,a^2,b,ab,a^2b,a^3b\},
\{a,a^3,b,ab,a^2b,a^3b\},\\
&\{a,a^2,a^3,b,ab,a^2b,a^3b\}.
\end{align*}
And $\Cay(T_8,S)$ is disconnected if $S$ is one of the following:
\begin{align*}
&\emptyset, \{a\}, \{a^2\}, \{b\}, \{a,a^2\}, \{a,a^3\}, \{a^2,b\},\{b,a^2b\}, \{a,a^2,a^3\}, \{a^2,b,a^2b\}.
\end{align*}
Thus, there are exactly $36$ dicirculant digraphs of order $8$ up to isomorphism, in which $26$ are connected.
\qed\end{proof}

In Table \ref{en-table1}, we list the number of connected dicirculant digraphs of order $4p$ and out-degree $k$ up to isomorphism when $0\leq k\leq 4p-1$ and $2\leq p\leq11$.

\begin{table}[h]
\begin{center}
\caption{ The number of connected dicirculant digraphs~$(2 \leq p \leq 11)$}\label{en-table1}
\begin{tabular}{c|>{\centering\arraybackslash}m{12cm}|c}
\hline
$p$    & $\left(\mathcal{M'}_2, \ldots, \mathcal{M'}_{4 p-1}\right)$          &$\mathcal{N'}$\\
\hline
2   &(2,6,8,6,3,1)   &26\\
\hline
3   & (4, 17,38,53,54,41,24,12,4,1)   &248\\
\hline
5  & (4,26,109,318,734,1341,2005,2447,2448,2008,1351,756,352,143,
49,16,4,1)
&14112\\
\hline
7    & (4,36,223,999,3645,10832,26942,56604,101661,156837,208957,
241024,241025,208960,156851,101711,56727,27159,11124,3937,
1216,346,87,20,4,1)
&1616932\\
\hline
11  & (4,54,563,4391,27961,147551,663267,2574938,8744601,26208517,
69845535,166478474,356632986,689343259,1206179201,1915502418,
2766650996,3640164264,4368079300,4784024372,4784024373,
4368079303,3640164286,2766651130,1915503021,1206181242,
689348699,356644626,166498844,69874949,26243817,8779901,
2604352,683637,159191,33401,6432,1166,189,28,4,1)
&40002755244\\
\hline
\end{tabular}
\end{center}
\end{table}

\section{Conclusion}
In this paper, we calculate the number of dicirculant digraphs by employing the P{\'o}lya Enumeration Theorem, in which the number of connected dicirculant digraphs is counted by deleting the  number of circulants and other disconnected graphs. We also get the number of (connected) dicirculant digraphs $\Cay(T_{4p}, S)$ of out-degree $k$ for all $0\leq k\leq 4p-1$. Finally, we list the number of (connected) dicirculant digraphs $\Cay(T_{4p}, S)$ of out-degree $k~(0\leq k\leq 4p-1)$ for $2\leq p\leq 11$. We have known that DCI-group is an important property in enumerating Cayley graphs. So we would like to propose the following problem:
\begin{problem}
Characterize other DCI-groups such as generalized dicyclic groups, generalized dihedral groups and semi-dihedral groups. Then enumerate the Cayley digraphs on these groups.
\end{problem}

\section*{Declaration of competing interest}
The authors declare that they have no conflict of interest.
\section*{Data availability}
No data was used for the research described in the article.

\end{document}